\documentclass[11pt,reqno]{amsart}

\usepackage{graphicx}

\usepackage{xspace,tikz}
\usepackage{amsmath,amsthm,amsfonts,amssymb,latexsym,mathrsfs,color,extarrows}
\usepackage{hyperref}

\usepackage{pgfplots}
\pgfplotsset{compat=newest}
\usetikzlibrary{shapes.geometric,arrows,fit,matrix,positioning}
\tikzset
{
    treenode/.style = {circle, draw=black, align=center, minimum size=1cm},
    subtree/.style  = {isosceles triangle, draw=black, align=center, minimum height=0.5cm, minimum width=1cm, shape border rotate=90, anchor=north}
}

\newtheorem{theorem}{Theorem}%[section]
\newtheorem{corollary}[theorem]{Corollary}
\newtheorem{proposition}[theorem]{Proposition}

\newtheorem{lemma}[theorem]{Lemma}
\newtheorem{definition}[theorem]{Definition}
\newtheorem{example}[theorem]{Example}
\newtheorem*{STG}{Symmetric transformation of grammar}

\newcommand{\XX}{{\rm X\,}}

\newcommand{\basc}{{\rm basc\,}}

\newcommand{\suc}{{\rm suc\,}}

\newcommand{\plat}{{\rm plat\,}}

\newcommand{\des}{{\rm des\,}}

\newcommand{\exc}{{\rm exc\,}}

\newcommand{\aexc}{{\rm aexc\,}}

\newcommand{\fix}{{\rm fix\,}}

\newcommand{\mdn}{\mathcal{D}}

\newcommand{\msn}{\mathfrak{S}_n}

\newcommand{\ms}{\mathfrak{S}}

\newcommand{\lrf}[1]{\lfloor #1\rfloor}

\newcommand{\mq}{\mathcal{Q}}
\newcommand{\mqn}{\mathcal{Q}_n}

\newcommand{\asc}{{\rm asc\,}}

\newcommand{\Eulerian}[2]{\genfrac{<}{>}{0pt}{}{#1}{#2}}
\newcommand{\Stirling}[2]{\genfrac{\{}{\}}{0pt}{}{#1}{#2}}

\title{Eulerian polynomials, Stirling permutations and increasing trees}
\author[S.-M.~Ma]{Shi-Mei Ma}
\address{School of Mathematics and Statistics,
        Northeastern University at Qinhuangdao,
         Hebei 066000, P.R. China}
\email{shimeimapapers@163.com (S.-M. Ma)}
\author[J.~Ma]{Jun Ma}
\address{Department of mathematics, Shanghai Jiao Tong University, Shanghai, P.R. China}
\email{majun904@sjtu.edu.cn}
\author{Jean Yeh}
\address{Departemnt of Mathematics, National Kaohsiung Normal University, Kaohsiung 82446, Taiwan}
\email{chunchenyeh@nknu.edu.tw}
\author{Yeong-Nan Yeh}
\address{Institute of Mathematics, Academia Sinica, Taipei, Taiwan}
\email{mayeh@math.sinica.edu.tw}
\begin{document}

\maketitle
\begin{abstract}
We study two generalizations of the $\gamma$-expansion
of Eulerian polynomials from the viewpoint of the decompositions of statistics.
We first present an expansion formula
of the trivariate Eulerian polynomials, which are the enumerators for the joint distribution of descents, big ascents and successions of permutations.
And then, inspired by the work of Chen and Fu on the
trivariate second-order Eulerian polynomials, we show the $e$-positivity of the multivariate $k$-th order Eulerian polynomials, which are the enumerators for the joint distribution of
ascents, descents and $j$-plateaux of $k$-Stirling permutations.
We provide combinatorial interpretations for the coefficients of these two expansions in terms of increasing trees.
\bigskip

\end{abstract}
\date{\today}
%\date{\today}
%%%%%%%%%%%%%%%%%%%%%%%%%%%%%%%%%%%%%%%%%%%%%%%%%%%%%%%%%%%%%%%%%%%%%%%%%%
%%%%%%%%%%%%%%%%%%%%%%%%%%%%%%%%%%%%%%%%%%%%%%%%%%%%%%%%%%%%%%%%%%%%%%%%%%
%%%%%%%%%%%%%%%%%%%%%%%%%%%%%%%%%%%%%%%%%%%%%%%%%%%%%%%%%%%%%%%%%%%%%%%%%%
%%%%%%%%%%%%%%%%%%%%%%%%%%%%%%%%%%%%%%%%%%%
\section{Introduction}
%%%%%%%%%%%%%%%%%%%%%%%%%%%%%%%%%%%%%%%%%%%
%%%%%%%%%%%%%%%%%%%%%%%%%%%%%%%%%%%%%%%%%%%%%%%%%%%%%%%%%%%%%%%%%%%%%%%%%%%%%%%%%
%%%%%%%%%%%%%%%%%%%%%%%%%%%%%%%%%%%%%%%%%%%%%%%%%%%%%%%%%%%%%%%%%%%%%%%%%%%%%%%%%
%%%%%%%%%%%%%%%%%%%%%%%%%%%%%%%%%%%%%%%%%%%%%%%%%%
%%%%%%%%%%%%%%%%%%%%%%%%%%%%%%%%%%%%%%%%%%%%
%\subsection{Preliminary}
%%%%%%%%%%%%%%%%%%%%%%%%%%%%%
%\hspace*{\parindent}
%%\subsection{%%%%%%%%%%%%%%%%%%%%%%%%%%%%%%%%%%%%%%%%%%%%%%%%%%%
%%%%%%%%%%%%%%%%%%%%%%%%%%%%%%%%%%%%%%%%%%%%%%%%%%%%%%%%%%%
%%%%%%%%%%%%%%%%%%%%%%%%%%%5%%%%
Let $f(x)=\sum_{i=0}^nf_ix^i$ be a symmetric polynomial of degree $n$, i.e.,
$f_i=f_{n-i}$ for any $0\leqslant i\leqslant n$. Then $f(x)$ can be expanded uniquely as
$$f(x)=\sum_{k=0}^{\lrf{{n}/{2}}}\gamma_kx^k(1+x)^{n-2k}.$$
We say that $f(x)$ is {\it $\gamma$-positive}
if $\gamma_k\geqslant 0$ for $0\leqslant k\leqslant \lrf{{n}/{2}}$ (see~\cite{Gal05,Lin21} for instance).
Notably, $\gamma$-positivity of a polynomial implies that its coefficients are symmetric and unimodal, and the
coefficients of $\gamma$-positive polynomials often have nice combinatorial interpretations.
We refer the reader to Athanasiadis's survey article~\cite{Athanasiadis17} for details.

This paper is devoted to generalize the $\gamma$-expansion of Eulerian polynomials.
There are two objects of this paper.
We first study the enumerators for the joint distribution of descents, big ascents and successions of permutations,
and then we study the enumerators for the joint distribution of
ascents, descents and $j$-plateaux of $k$-Stirling permutations.

Let $\msn$ denote the symmetric group of all permutations of $[n]$, where $[n]=\{1,2,\ldots,n\}$.
As usual, we write $\pi=\pi(1)\pi(2)\cdots\pi(n)\in\msn$.
A {\it descent} (resp.~{\it ascent, excedance}) of $\pi$ is an index $i\in[n-1]$
such that $\pi(i)>\pi(i+1)$ (resp.~$\pi(i)<\pi(i+1)$, $\pi(i)>i$). Let $\des(\pi)$ (resp.~$\asc(\pi)$, $\exc(\pi)$) denote the number of descents (resp.~ascents, excedances) of $\pi$.
It is well known that descents, ascents and excedances are equidistributed over the symmetric groups,
and their common enumerative polynomials are the {\it Eulerian polynomials} $A_n(x)$, i.e.,
$$A_n(x)=\sum_{\pi\in\msn}x^{\des(\pi)}=\sum_{\pi\in\msn}x^{\asc(\pi)}=\sum_{\pi\in\msn}x^{\exc(\pi)}.$$
Thus any other
statistic that is equidistributed with $\des$ or $\exc$ is called an Eulerian statistic.

There has been much work on the combinatorial expansions of Eulerian polynomials, see e.g.~\cite{Chow08,Petersen07,Zhuang17}.
Let $\Stirling{n}{k}$ be the {\it Stirling number of the second kind}, which counts the number of partitions
of $[n]$ into $k$ nonempty blocks.
A famous combinatorial expansion of the Eulerian polynomial is the Frobenius formula:
\begin{equation}\label{Frobenius}
xA_n(x)=\sum_{k=1}^nk!\Stirling{n}{k}x^k(1-x)^{n-k}~\text{for any $n\geqslant 1$}.
\end{equation}
We say that an index $i$ is a {\it double descent} of $\pi\in\msn$ if $\pi(i-1)>\pi(i)>\pi(i+1)$, where $\pi(0)=\pi(n+1)=0$.
Another famous combinatorial expansion was first established by Foata and Sch\"utzenberger~\cite{Foata70}, which says that
\begin{equation}\label{Anx-gamma}
A_n(x)=\sum_{i=0}^{\lrf{(n-1)/2}}\gamma_{n,i}x^i(1+x)^{n-1-2i},
\end{equation}
where $\gamma_{n,i}$ is the number of permutations $\pi\in\msn$ which have no double descents and $\des(\pi)=i$.
The $\gamma$-expansion~\eqref{Anx-gamma}, along several multivariate refinements and $q$-analogues, were frequently discovered in the past decades, see~\cite{Athanasiadis17,Ma19,Zhuang17} and references therein.
For example, by using the theory of enriched $P$-partitions, Stembridge~\cite[Remark 4.8]{Stembridge97} showed that
\begin{equation}\label{AnxWni}
A_n(x)=\frac{1}{2^{n-1}}\sum_{i=0}^{\lrf{(n-1)/2}}4^iP(n,i)x^i(1+x)^{n-1-2i},
\end{equation}
where $P(n,i)$ is the the number of permutations in $\msn$ with $i$ {\it interior peaks}, i.e., the indices $i\in\{2,\ldots,n-1\}$ such that $\pi(i-1)<\pi(i)>\pi(i+1)$.
By using modified Foata-Strehl action~\cite{Branden08}, one can see that the expansion~\eqref{AnxWni} is equivalent to~\eqref{Anx-gamma}.

In the past decades, the bijections between permutations and increasing trees are repeatedly discovered (see~\cite[Section~1.5]{Stanley2011} for instance).
In this paper, by using the theory of context-free grammars, we shall present two generalizations
of~\eqref{Anx-gamma}, and the combinatorial interpretations for the coefficients of the generalized expansions are given in terms of increasing trees.
The organization of this paper is as follows. In the next section, we give a short survey about the theory of context-free grammars. In Section~\ref{Section2}, we study the joint distribution of descents, big ascents and successions of permutations.
In Section~\ref{Section3}, we study the joint distribution of descents, ascents and $j$-plateaux of $k$-Stirling permutations.
%%%%%%%%%%%%%%%%%%%%%%%%%%%%%%%%%%%%%%%%%%%%%%%%%%%%%%%%%%%%%%%%%%%%%%%%%%
%%%%%%%%%%%%%%%%%%%%%%%%%%%%%%%%%%%%%%%%%%%
\section{The theory of context-free grammars}\label{grammarsection}
%%%%%%%%%%%%%%%%%%%%%%%%%%%%%%%%%%%%%%%%%%%
%%%%%%%%%%%%%%%%%%%%%%%%%%%%%%%%%%%%%%%%%%%%%%%%%%%%%%%%%%%%%%%%%%%%%%%%%%%%%%%%%
%%%%%%%%%%%%%%%%%%%%%%%%%%%%%%%%%%%%%%%%%%%%%%%%%%%%%%%%%%%%%%%%%%%%%%%%%%%%%%%%%
%%%%%%%%%%%%%%%%%%%%%%%%%%%%%%%%%%%%%%%%%%%%%%%%%%
%%%%%%%%%%%%%%%%%%%%%%%%%%%%%%%%%%%%%%%%%%%%
For an alphabet $A$, let $\mathbb{Q}[[A]]$ be the rational commutative ring of formal power
series in monomials formed from letters in $A$. Following Chen~\cite{Chen93}, a {\it context-free grammar} (also known as {\it Chen's grammar}) over
$A$ is a function $G: A\rightarrow \mathbb{Q}[[A]]$ that replaces each letter in $A$ by a formal function over $A$.
The formal derivative $D_G$ with respect to $G$ satisfies the derivation rule:
$$D_G(u+v)=D_G(u)+D_G(v),~D_G(uv)=D_G(u)v+uD_G(v).$$

Context-free grammar is a powerful tool for studying
exponential structures in combinatorics. We refer the reader to~\cite{Chen17,Chen22,Dumont96,Lin21,Ma19} for further
information. For example, for a large number of classical
combinatorial polynomials, Dumont~\cite{Dumont96} gave
a natural explanation of the mysterious coincidences between calculus
and enumeration.
Let us now recall a grammatical interpretation of Eulerian polynomials.
\begin{proposition}[{\cite[Section~2.1]{Dumont96}}]\label{Dumont96}
Let $G_1=\{x\rightarrow xy, y\rightarrow xy\}$.
Then
\begin{equation*}
D_{G_1}^n(x)=x\sum_{k=0}^{n-1}\Eulerian{n}{k}x^{k}y^{n-k}\quad\textrm{for $n\geqslant 1$}.
\end{equation*}
Setting $y=1$, one has $D_{G_1}^n(x)\mid_{y=1}=xA_n(x)$.
\end{proposition}

The following two definitions will be used repeatedly in our discussion.
\begin{definition}[{\cite{Chen17}}]
A grammatical labeling is an assignment of the underlying elements of a combinatorial structure
with variables, which is consistent with the substitution rules of a grammar.
\end{definition}

\begin{definition}[{\cite{Ma19}}]
A change of grammars is a substitution method in which the original grammar is replaced with functions of the other grammar.
\end{definition}

In~\cite{Ma19}, the change of grammars method was introduced and the power of it is to obtain recurrences and combinatorial interpretations of the $\gamma$-coefficients and
partial $\gamma$-coefficients of various polynomials.
Recently, by using the change of grammars method, Lin etal.~\cite{Lin21} proved the partial-$\gamma$-positivity of trivariate enumerative polynomials of Stirling multipermutations.

Let us first give some examples of the change of grammars.
Consider the change of variables:
\begin{equation}\label{change-grammars01}
\left\{
  \begin{array}{ll}
    u=xy, &  \\
    v=x+y, &
  \end{array}
\right.
\end{equation}
The grammar $G_1=\{x\rightarrow xy, y\rightarrow xy\}$ is transformed into a new grammar
\begin{equation}\label{G2-def}
G_2=\{u\rightarrow uv, v\rightarrow 2u\}.
\end{equation}
Since $D_{G_1}^n(x)=D_{G_2}^{n-1}(u)$ for $n\geqslant 1$, it immediately follows that $A_n(x)$ is $\gamma$-positive (see~\cite{Ma19} for details).
In the following, we provide another illustration of the change of grammars method.

\noindent{\bf A grammatical proof of the Frobenius formula~\eqref{Frobenius}:}
\begin{proof}
Let $G_1$ be the grammar given in Proposition~\ref{Dumont96}. Set $u=y-x$. Observe that $$D_{G_1}(x)=x(u+x),~D_{G_1}(u)=0.$$
Then we get a new grammar
$G_3=\{x\rightarrow x(u+x),~u\rightarrow 0\}$.
Note that
$$D_{G_3}(x)=x(u+x), D_{G_3}^2(x)=(u+x)(xu+2x^2),~D_{G_3}^3(x)=(u+x)(xu^2+6x^2u+6x^3).$$
For $n\geqslant 1$, assume that $D_{G_3}^n(x)=(u+x)\sum_{k=1}^nE_{n,k}x^ku^{n-k}$.
Then
\begin{align*}
D_{G_3}^{n+1}(x)&=D_{G_3}(D_{G_3}^n(x))\\
&=D_{G_3}\left((u+x)\sum_{k=1}^nE_{n,k}x^ku^{n-k}\right)\\
&=(u+x)\sum_{k}E_{n,k}\left(kx^ku^{n-k+1}+(1+k)x^{k+1}u^{n-k}\right),
\end{align*}
which yields the recurrence relation $E_{n+1,k}=k(E_{n,k}+E_{n,k-1})$.
From~\cite[A019538]{Sloane}, we see that $E_{n,k}$ satisfy the same recurrence relation
and initial conditions as  $k!\Stirling{n}{k}$, so they agree.
Setting $y=1$, we obtain $u=1-x$. Therefore, we find that
\begin{equation*}
D_{G_3}^{n}(x)\mid_{y=1}=\sum_{k=1}^nk!\Stirling{n}{k}x^k(1-x)^{n-k}.
\end{equation*}
Comparing with Proposition~\ref{Dumont96} yields the desired formula~\eqref{Frobenius}, and so the proof is complete.
\end{proof}

Consider the following change of variables:
\begin{equation}\label{change-grammars}
\left\{
  \begin{array}{ll}
    u=2xy, &  \\
    v=x+y, &
  \end{array}
\right.
\end{equation}
the grammar $G_1=\{x\rightarrow xy, y\rightarrow xy\}$ is transformed into a new grammar
\begin{equation}\label{G3-def}
G_4=\{u\rightarrow uv, v\rightarrow u\}.
\end{equation}
The grammar $G_4$ was introduced by Dumont~\cite{Dumont96} when he studied Andr\'e polynomials.

A {\it rooted tree} of order $n$ with the vertices labelled $1,2,\ldots,n$, is an increasing tree if the node
labelled $1$ is distinguished as the root, and for each $2\leqslant i\leqslant n$, the labels of the nodes in the unique
path from the root to the node labelled $i$ form an increasing sequence.
An {\it increasing tree} on $\{0,1,2,\ldots,n\}$ is a rooted tree
with vertex set $\{0,1,2,\ldots,n\}$ in which the labels of the vertices are increasing along any path from the root 0 to a leaf.
The {\it degree} of a vertex is referred to the number of its children.
A {\it 0-1-2 increasing tree} is an increasing tree in which the degree of any vertex is at most two.
Let $\ell(T)$ denote the number of leaves of a tree and let $r(T)$ denote the number of vertices with degree $1$.
The {\it Andr\'e polynomials} are defined by
\begin{equation}\label{Enuv}
E_n(u,v)=\sum_{T}u^{\ell(T)}v^{r(T)},
\end{equation}
where the sum ranges over 0-1-2 increasing trees on $\{0,1,2,,n\}$.
Dumont showed that
\begin{equation}\label{DumontG3}
D_{G_4}^n(u)=E_n(u,v).
\end{equation}
Following Chen and Fu~\cite{Chen17}, the following grammatical labeling leads to the relation~\eqref{DumontG3}:
a leaf is labeled by $u$, a vertex of degree 1 is labeled by $v$ and a vertex of degree 2 is labeled by 1.

A {\it plane tree} is a rooted tree in which the children of each vertex are linearly ordered (from left to right, say).
Increasing plane trees are also called plane recursive trees, see~\cite{Janson08,Janson11}. A {\it 0-1-2 increasing plane tree} on $[n]$
is an increasing plane tree for which each vertex has degree at most two.
From the above discussion, we see that for $n\geqslant 1$,
\begin{equation}\label{xAnx}
xA_n(x)=D_{G_1}^n(x)\mid_{y=1}=D_{G_2}^{n-1}(u)\mid_{y=1}=\sum_{i=1}^{\lrf{(n+1)/2}}\gamma_{n,i-1}x^i(1+x)^{n+1-2i}.
\end{equation}
By using grammatical labeling,
Chen and Fu~\cite{Chen22} found the following result.
\begin{proposition}[{\cite[Theorem~3.1]{Chen22}}]
Let $\gamma_{n,i-1}$ be the coefficient defined by~\eqref{xAnx}, where $1\leqslant i\leqslant \lrf{(n+1)/2}$.
Then $\gamma_{n,i-1}$ equals the
number of 0-1-2 increasing plane trees on $[n]$ with $i$ leaves.
\end{proposition}

Let $\XX_n=\{x_1,x_2,\ldots,x_n\}$ be a set of commuting variables.
Define $$S_n(x)=\prod_{i=1}^n(x-x_i)=\sum_{k=0}^n{(-1)}^ke_{k}x^{n-k}.$$
Then the {\it $k$-th
elementary symmetric function} associated with $\XX_n$ is defined by $$e_k=\sum_{1\leqslant i_1<i_2<\cdots<i_k\leqslant n}x_{i_1}x_{i_2}\cdots x_{i_k}.$$
In particular, $e_0=1$, $e_1=\sum_{i=1}^nx_i$ and $e_n=x_1x_2\cdots x_n$.
A function $f(x_1,x_2,\ldots)\in \mathbb{R}[x_1,x_2,\ldots]$ is said to be {\it symmetric} if it is invariant under
any permutation of its indeterminates.
We say that a symmetric function is {\it $e$-positive} if it can be written as a nonnegative linear combination of elementary symmetric
functions.
In~\cite{Chen22}, Chen and Fu introduced a new change of variables:
\begin{equation}\label{grammars02}
\left\{
  \begin{array}{ll}
    u=x+y+z, &  \\
    v=xy+yz+zx, &\\
    w=xyz.
  \end{array}
\right.
\end{equation}
By using~\eqref{grammars02}, they discovered the new grammar~\eqref{G8def}, and then they proved the $e$-positivity of trivariate
second-order Eulerian polynomials (see~\cite[Section~4]{Chen22}).
As a unified extension of~\eqref{change-grammars01},~\eqref{change-grammars} and~\eqref{grammars02}, we now introduce a definition. A special case is discussed in Section~\ref{Section3}.
\begin{STG}
Let $G$ be the grammar defined by
\begin{equation*}\label{type-grammar}
G=\{x_1\rightarrow f_1(x_1,x_2,\ldots, x_n), x_2\rightarrow f_2(x_1,x_2,\ldots,x_n),\cdots, x_n\rightarrow f_n(x_1,x_2,\ldots,x_n)\}.
\end{equation*}
Suppose that $D_{G}^n\left(F(x_1,x_2,\ldots, x_n)\right)$ are symmetric functions.
The symmetric transformation of $G$ is defined by $u_j=g_j(e_1,e_2,\ldots,e_n)$, where $F$ and $g_j$ are functions, $e_{\ell}$ are the $\ell$-th
elementary symmetric functions associated with $\{x_1,x_2,\ldots,x_n\}$, and $1\leqslant j,\ell\leqslant n$.
\end{STG}

%%%%%%%%%%%%%%%%%%%%%%%%%%%%%%%%%%%%%%%%%%%%%%%%%%%%%%%%%%%%%%%%%%%%%%%%%%
%%%%%%%%%%%%%%%%%%%%%%%%%%%%%%%%%%%%%%%%%%%
\section{Trivariate Eulerian polynomials}\label{Section2}
%%%%%%%%%%%%%%%%%%%%%%%%%%%%%%%%%%%%%%%%%%%
%%%%%%%%%%%%%%%%%%%%%%%%%%%%%%%%%%%%%%%%%%%%%%%%%%%%%%%%%%%%%%%%%%%%%%%%%%%%%%%%%
%%%%%%%%%%%%%%%%%%%%%%%%%%%%%%%%%%%%%%%%%%%%%%%%%%%%%%%%%%%%%%%%%%%%%%%%%%%%%%%%%
%%%%%%%%%%%%%%%%%%%%%%%%%%%%%%%%%%%%%%%%%%%%%%%%%%
%%%%%%%%%%%%%%%%%%%%%%%%%%%%%%%%%%%%%%%%%%%%
%%%%%%%%%%%%%%%%%%%%%%%%%%%%%%%%%%%%%%%%%%%%
\subsection{Preliminary}
%%%%%%%%%%%%%%%%%%%%%%%%%%%%%
\hspace*{\parindent}
%%\subsection{%%%%%%%%%%%%%%%%%%%%%%%%%%%%%%%%%%%%%%%%%%%%%%%%%%%
%%%%%%%%%%%%%%%%%%%%%%%%%%%%%%%%%%%%%%%%%%%%%%%%%%%%%%%%%%%

The enumeration of permutations according to the number of successions was initiated by Riordan~\cite{Riordan45}.
A {\it succession} of $\pi\in\msn$ is an index $k\in [n-1]$ such that $\pi(k+1)=\pi(k)+1$.
Let $P(n,r,s)=\#\{\pi\in\msn: \asc(\pi)=r,~\suc(\pi)=s\}$, where $\suc(\pi)$ is the number of successions of $\pi$.
Roselle~\cite[Eq.~(2.1)]{Roselle68} proved that
$$P(n,r,s)=\binom{n-1}{s}P(n-s,r-s,0).$$
Denote by $P^*(n,r)$ the number of permutations of $\msn$ with $r$ ascents, no successions and $\pi(1)>1$.
Let $P^*_n(x)=\sum_{k=0}^{n-1}P^*(n,r)x^r$.
According to~\cite[Eq.~(4.3)]{Roselle68}, one has
\begin{equation}\label{EGF1}
\sum_{n=0}^{\infty}P^*_n(x)\frac{z^n}{n!}=\frac{1-x}{\mathrm{e}^{xz}-x\mathrm{e}^{z}}.
\end{equation}
The reader is referred to~\cite{Brenti18,Clarke97,Diaconis14,Mansour16} for recent progresses in studies of succession statistics.
In particular, Diaconis, Evans and Graham~\cite{Diaconis14} found the following remarkable result, which was generalized soon by Brenti and Marietti~\cite{Brenti18}.
% They presented
%three different proofs of it, including an enumerative proof, a Markov chain proof and a bijective proof.
\begin{proposition}[{\cite{Diaconis14}}]\label{PropDiaconis}
For all $I\subseteq [n-1]$, we have
\begin{align*}
&\#\{\pi\in\msn: \{k\in [n-1]: \pi(k+1)=\pi(k)+1\}= I\}\\
&=\#\{\pi\in\msn: \{k\in [n-1]: \pi(k)=k\}= I\}.
\end{align*}
\end{proposition}

A {\it fixed point} of $\pi\in\msn$ is an index $k\in [n]$ such that $\pi(k)=k$. Let $\fix(\pi)$ be the number of fixed points of $\pi$.
There is a wealth of literature on the joint distributions of Eulerian and fixed point statistics, see e.g.~\cite{Brenti90,Clarke97,Foata08,Zeng16}.
Motivated by Proposition~\ref{PropDiaconis}, we shall study the joint distribution of Eulerian and succession statistics
from the viewpoint of the decomposition of the ascent statistic, i.e., write $\asc$ as a sum of the numbers of successions and big ascents.
%%%%%%%%%%%%%%%%%%%%%%%%%%%%%%%%%%%%%%%%%%%%%%%%%%
%%%%%%%%%%%%%%%%%%%%%%%%%%%%%%%%%%%%%%%%%%%%
\subsection{Main results}
%%%%%%%%%%%%%%%%%%%%%%%%%%%%%
\hspace*{\parindent}
%%\subsection{%%%%%%%%%%%%%%%%%%%%%%%%%%%%%%%%%%%%%%%%%%%%%%%%%%%
%%%%%%%%%%%%%%%%%%%%%%%%%%%%%%%%%%%%%%%%%%%%%%%%%%%%%%%%%%%

The number of {\it big ascents} of $\pi\in\msn$ is defined by
\begin{align*}
\basc(\pi)&=\#\{i\in [n-1]:~\pi(i+1)\geqslant \pi(i)+2\}.
\end{align*}
Note that $\asc(\pi)=\suc(\pi)+\basc(\pi)$.
Consider the {\it trivariate Eulerian polynomials}
\begin{equation}\label{AnxysDef}
A_n(x,y,s)=\sum_{\pi\in\msn}x^{\basc(\pi)}y^{\des(\pi)}s^{\suc(\pi)}.
\end{equation}
Below are the polynomials $A_n(x,y,s)$ for $0\leqslant n\leqslant 5$:
\begin{align*}
A_0(x,y,s)&=A_1(x,y,s)=1,~
A_2(x,y,s)=s+y,\\
A_3(x,y,s)&=(s+y)^2+2xy,~
A_4(x,y,s)=(s+y)^3+6xy(s+y)+2xy(x+y),\\
A_5(x,y,s)&=(s+y)^4+12xy(s+y)^2+8xy(s+y)(x+y)+2xy(x+y)^2+16x^2y^2.
\end{align*}
Note that $A_n(x)=A_n(x,1,x)=A_n(1,x,1)$.
The main result of this section is the following.
\begin{theorem}\label{mainthm01}
Let $A_n(x,y,s)$ be the trivariate Eulerian polynomials defined by~\eqref{AnxysDef}.
\begin{itemize}
  \item [$(i)$] We have
\begin{equation}\label{A-EGF}
A(x,y,s;z)=\sum_{n=0}^{\infty}A_{n+1}(x,y,s)\frac{z^n}{n!}=\mathrm{e}^{z(y+s)}\left(\frac{y-x}{y\mathrm{e}^{xz}-x\mathrm{e}^{yz}}\right)^2;
\end{equation}
  \item [$(ii)$] For any $n\geqslant 0$, we have
  \begin{equation}\label{Anxys-gamma}
A_{n+1}(x,y,s)=\sum_{i=0}^n(s+y)^i\sum_{j=0}^{\lrf{(n-i)/2}}\gamma_{n,i,j}(2xy)^j(x+y)^{n-i-2j},
\end{equation}
where the numbers $\gamma_{n,i,j}$ satisfy the recurrence relation
\begin{equation}\label{gammanij-recu}
\gamma_{n+1,i,j}=\gamma_{n,i-1,j}+(1+i)\gamma_{n,i+1,j-1}+j\gamma_{n,i,j}+(n-i-2j+2)\gamma_{n,i,j-1},
\end{equation}
with the initial conditions $\gamma_{0,0,0}=1$ and $\gamma_{0,i,j}=0$ for $(i,j)\neq (0,0)$;
  \item [$(iii)$] The number $\gamma_{n,i,j}$ equals the
number of 0-1-2 increasing rooted forests on $\{0,1,\ldots,n\}$ with
$i+j$ leaves, among which $i$ leaves are children of the root, where a 0-1-2 increasing rooted forest on $\{0,1,2,\ldots,n\}$ is a
tree on $\{0,1,\ldots,n\}$ with the restriction that the root 0 has only children with degree 0 or 1,
and the degree of any other vertex is at most two.
\end{itemize}
\end{theorem}

The reader is referred to Fig.~\ref{fig:simple} for an example of a 0-1-2 increasing rooted forest, where the root 0 has three 0-1-2 subtrees.
Setting $y=1$ in~\eqref{Anxys-gamma}, we see that
$$\sum_{\pi\in\ms_{n+1}}x^{\basc(\pi)}s^{\suc(\pi)}=\sum_{i=0}^n(1+s)^i\sum_{j=0}^{\lrf{(n-i)/2}}\gamma_{n,i,j}(2x)^j(1+x)^{n-i-2j},$$
and it reduces to~\eqref{Anx-gamma} when $s=x$.

We say that $\pi$ is a {\it derangement} if it has no fixed points.
Let $\mdn_n$ be the set of derangements in $\msn$.
The {\it derangement polynomials} are defined by $d_n(x)=\sum_{\pi\in\mdn_n}x^{\exc(\pi)}$.
The generating function of $d_n(x)$ is given as follows (see~\cite[Proposition~6]{Brenti90}):
\begin{equation}\label{dxz-EGF}
d(x,z)=\sum_{n=0}^\infty d_n(x)\frac{z^n}{n!}=\frac{1-x}{\mathrm{e}^{xz}-x\mathrm{e}^{z}}.
\end{equation}
Comparing~\eqref{EGF1} with~\eqref{dxz-EGF}, we see that $P^*_n(x)=d_n(x)$.
A {\it anti-excedance} of $\pi\in\msn$ is an index $i\in[n-1]$ such that $\pi(i)<i$.
Let $\aexc(\pi)$ denote the number of anti-excedances of $\pi$.
Clearly, $\exc(\pi)+\aexc(\pi)+\fix(\pi)=n$.
Define
$$C_n(x,y,s)=\sum_{\pi\in\msn}x^{\exc(\pi)}y^{\aexc(\pi)}s^{\fix(\pi)}.$$
In particular, $d_n(x)=C_n(x,1,0)$ and $A_n(x)=C_n(x,1,1)$.
It is well known (see~\cite{Foata70,Zeng16}) that
\begin{equation}\label{dxys-EGF}
C(x,y,s;z)=\sum_{n=0}^{\infty}C_n(x,y,s)\frac{z^n}{n!}=\frac{(y-x)\mathrm{e}^{sz}}{y\mathrm{e}^{xz}-x\mathrm{e}^{yz}}.
\end{equation}
Let $A_n(x,y)$ be the bivariate Eulerian polynomials defined by
\begin{equation}\label{Anxy02}
\sum_{n=0}^\infty{A}_n(x,y)\frac{z^n}{n!}=\frac{(y-x)\mathrm{e}^{yz}}{y\mathrm{e}^{xz}-x\mathrm{e}^{yz}},
\end{equation}
which can also be defined by $${A}_n(x,y)=\sum_{\pi\in\msn}x^{\asc(\pi)}y^{\des(\pi)+1},~{A}_0(x,y)=1.$$
%\begin{equation*}\label{Ankx-deff}
%A(x;z)=\sum_{n=0}^\infty A_{n}(x)\frac{z^n}{n!}=\frac{x-1}{x-e^{(x-1)z}}.
%\end{equation*}
%Equivalently,
Combining~\eqref{A-EGF},~\eqref{dxys-EGF} and~\eqref{Anxy02}, we get the following result.
\begin{corollary}
For $n\geqslant 0$, we have
\begin{equation*}
A_{n+1}(x,y,s)=\sum_{i=0}^n\binom{n}{i}{A}_i(x,y)C_{n-i}(x,y,s),
\end{equation*}
In particular, $$A_{n+1}(x,1,0)=\sum_{i=0}^n\binom{n}{i}A_i(x)d_{n-i}(x),~A_{n+1}(x,1,1)=\sum_{i=0}^n\binom{n}{i}A_i(x)A_{n-i}(x).$$
\end{corollary}

Let $\gamma_{n,i,j}$ be the numbers defined by~\eqref{gammanij-recu}.
Now define
$$\gamma_n(x,y)=\sum_{i=0}^n\sum_{j=0}^{\lrf{(n-i)/2}}\gamma_{n,i,j}x^iy^j,~\gamma(x,y;z)=\sum_{n=0}^\infty \gamma_n(x,y) \frac{z^n}{n!}.$$
We end this subsection by giving the following result.
\begin{proposition}
We have
\begin{equation}\label{gammaxyz}
\gamma(x,y;z)=e^{z(x-1)}\left(\frac{\sqrt{2y-1}\sec\left(\frac{z}{2}\sqrt{2y-1}\right)}
{\sqrt{2y-1}-\tan\left(\frac{z}{2}\sqrt{2y-1}\right)}\right)^2.
\end{equation}
\end{proposition}
\begin{proof}
Multiplying both sides of~\eqref{gammanij-recu} by $x^iy^j$ and summing over all $i,j$, we obtain
\begin{equation*}\label{recugamma}
\gamma_{n+1}(x,y)=(x+ny)\gamma_n(x,y)+y(1-x)\frac{\partial}{\partial x}\gamma_n(x,y)+y(1-2y)\frac{\partial}{\partial y}\gamma_n(x,y),
\end{equation*}
which can be rewritten as
\begin{equation}\label{EGFgamma}
(1-yz)\frac{\partial}{\partial x}\gamma(x,y;z)=x\gamma(x,y;z)+y(1-x)\gamma(x,y;z)+y(1-2y)\gamma(x,y;z).
\end{equation}
It is routine to check that the generating function
$$\widehat{\gamma}(x,y;z)=e^{z(x-1)}\left(\frac{\sqrt{2y-1}\sec\left(\frac{z}{2}\sqrt{2y-1}\right)}
{\sqrt{2y-1}-\tan\left(\frac{z}{2}\sqrt{2y-1}\right)}\right)^2$$
satisfies~\eqref{EGFgamma}. Note that this generating function gives $\widehat{\gamma}(x,y;0)=1$ and $\widehat{\gamma}(x,0;0)=\mathrm{e}^{xz}$.
Hence $\gamma(x,y;z)=\widehat{\gamma}(x,y;z)$.
\end{proof}

When $x=1$, the explicit formula~\eqref{gammaxyz} reduces to the exponential generating function of the descent polynomials of simsun permutations,
which was obtained by Chow and Shiu~\cite[Theorem~1]{Chow11}. Let $E_n(u,v)$ be the Andr\'e polynomials defined by~\eqref{Enuv}.
It should be noted that $$\gamma(1,y;z)=1+\sum_{n=1}^\infty \frac{1}{y}E_n(y,1)\frac{z^n}{n!},$$
and an equivalent explicit formula of $\gamma(1,y;z)$ has been obtained by
Foata and Sch\"utzenberger~\cite{Foata73} as well as Foata and Han~\cite[Section~7]{Foata01}.
%%%%%%%%%%%%%%%%%%%%%%%%%%%%%%%%%%%%%%%%%%%%%%%%%%%%%%%%%%%
%%%%%%%%%%%%%%%%%%%%%%%%%%%%%%%%%%%%%%%%%%%%
\subsection{A key Lemma}
%%%%%%%%%%%%%%%%%%%%%%%%%%%%%
\hspace*{\parindent}
%%\subsection{%%%%%%%%%%%%%%%%%%%%%%%%%%%%%%%%%%%%%%%%%%%%%%%%%%%
%%%%%%%%%%%%%%%%%%%%%%%%%%%%%%%%%%%%%%%%%%%%%%%%%%%%%%%%%%%

%%%%%%%%%%%%%%%%%%%%%%%%%%%%%%%%%%%%%%%%%%%
\begin{lemma}\label{lemmaLM}
If
\begin{equation}\label{LMxys-G}
G_5=\{L\rightarrow Ly,M\rightarrow Ms,s\rightarrow xy,x\rightarrow xy,y\rightarrow xy\},
\end{equation}
then we have
\begin{equation}\label{derangment-grammarLMA}
D_{G_5}^n(LM)=LMA_{n+1}(x,y,s).
\end{equation}
\end{lemma}
\begin{proof}
Here we introduce a grammatical labeling of $\pi=\pi(1)\pi(2)\cdots\pi(n)\in \msn$ as follows:
\begin{itemize}
\item [\rm ($i$)]Put a superscript label $L$ at the front of $\pi$;
\item [\rm ($ii$)]Put a superscript label $M$ right after the maximum entry $n$;
  \item [\rm ($iii$)]If $i$ is a big ascent, then put a superscript label $x$ right after $\pi(i)$;
 \item [\rm ($iv$)]If $i$ is a descent and $\pi(i)\neq n$, then put a superscript label $y$ right after $\pi(i)$;
  \item [\rm ($v$)]If $\pi(n)\neq n$, then put a superscript label $y$ at the end of $\pi$;
\item [\rm ($vi$)]If $i$ is a succession, then put a superscript label $s$ right after $\pi(i)$.
\end{itemize}
The weight of $\pi$ is defined to be the product of its labels.
Thus the weight of $\pi$ is given by $$w(\pi)=LMx^{\basc(\pi)}y^{\des(\pi)}s^{\suc(\pi)}.$$
Note that $\ms_1=\{^L1^M\}$ and $\ms_2=\{^L1^s2^M, ^L2^M1^y\}$. Note that $D_{G_5}(LM)=LM(s+y)$.
Hence the weight of the element in $\ms_1$ is $LM$ and the sum of weights of the elements in $\ms_2$ is given by $D_{G_5}(LM)$.
Suppose we get all labeled permutations in $\pi\in\ms_{n-1}$, where $n\geqslant 2$. Let
$\widehat{{\pi}}$ be obtained from $\pi\in\ms_{n-1}$ by inserting the entry $n$.
There are six cases to label $n$ and relabel some elements of $\pi$.
The changes of labeling are illustrated as follows:
$$ ^L\pi(1)\cdots (n-1)^M\cdots   \mapsto  ^Ln^M\pi(1)\cdots (n-1)^y\cdots ;$$
$$ ^L\pi(1)\cdots (n-1)^M\cdots   \mapsto ^L\pi(1)\cdots (n-1)^sn^M\cdots;$$
$$ \cdots\pi(i)^x\cdots (n-1)^M\cdots   \mapsto \cdots\pi(i)^xn^M\cdots (n-1)^y\cdots ;$$
$$ \cdots\pi(i)^y\pi(i+1)\cdots(n-1)^M\cdots   \mapsto \cdots\pi(i)^xn^M\pi(i+1)\cdots(n-1)^y\cdots  ;$$
$$ \cdots (n-1)^M\cdots\pi(n-1)^y   \mapsto  \cdots (n-1)^y\cdots\pi(n-1)^xn^M  ;$$
$$ \cdots \pi(i)^s\pi(i+1)\cdots(n-1)^M\cdots   \mapsto \cdots \pi(i)^xn^M\pi(i+1)\cdots (n-1)^y\cdots.$$
In each case, the insertion of $n$ corresponds to one substitution rule in $G_5$. By induction,
it is routine to check that the action of the formal derivative $D_{G_5}$ on the set of weights of permutations in $\ms_{n-1}$ gives the set of weights of permutations in
$\ms_{n}$. This completes the proof of~\eqref{derangment-grammarLMA}.
\end{proof}
%%%%%%%%%%%%%%%%%%%%%%%%%%%%%%%%%%%%%%%%%%%%%%%%%%%%%%%%%%%
%%%%%%%%%%%%%%%%%%%%%%%%%%%%%%%%%%%%%%%%%%%
%%%%%%%%%%%%%%%%%%%%%%%%%%%%%%%%%%%%%%%%%%%%%%%%%%%%%%%%%%%
%%%%%%%%%%%%%%%%%%%%%%%%%%%%%%%%%%%%%%%%%%%%
\subsection{Proof of Theorem~\ref{mainthm01}}
%%%%%%%%%%%%%%%%%%%%%%%%%%%%%
\hspace*{\parindent}
%%\subsection{%%%%%%%%%%%%%%%%%%%%%%%%%%%%%%%%%%%%%%%%%%%%%%%%%%%
%%%%%%%%%%%%%%%%%%%%%%%%%%%%%%%%%%%%%%%%%%%%%%%%%%%%%%%%%%%
%%%%%%%%%%%%%%%%%%%%%%%%%%%%%%%%%%%%%%%%%%%

\quad (A) Let $G_5$ be the grammar given by~\eqref{LMxys-G}.
From~\ref{derangment-grammarLMA}, we see that there exist nonnegative integers $a_{n,i,j}$ such that
$D_{G_5}^n(LM)=LM\sum_{i,j=0}^na_{n,i,j}x^iy^js^{n-i-j}$.
Then we have
\begin{align*}
&D_{G_5}\left(D_{G_5}^n(LM)\right)\\
&=LM\sum_{i,j=0}^na_{n,i,j}\left(x^iy^{j+1}s^{n-i-j}+x^iy^js^{n+1-i-j}\right)+\\
&LM\sum_{i,j=0}^na_{n,i,j}\left(ix^iy^{j+1}s^{n-i-j}+jx^{i+1}y^{j}s^{n-i-j}+(n-i-j)x^{i+1}y^{j+1}s^{n-1-i-j}\right).
\end{align*}
Comparing the coefficients of $LMx^iy^js^{n+1-i-j}$ in both sides of the above expression, we get
\begin{equation}\label{anij-recu}
a_{n+1,i,j}=a_{n,i,j}+(1+i)a_{n,i,j-1}+ja_{n,i-1,j}+(n-i-j+2)a_{n,i-1,j-1}.
\end{equation}
Multiplying both sides of~\eqref{anij-recu} by $x^iy^js^{n+1-i-j}$ and summing over all $i,j$, we obtain
\begin{equation}\label{Anxys-recu}
A_{n+2}(x,y,s)=(s+y)A_{n+1}(x,y,s)+xy\left(\frac{\partial}{\partial x}+\frac{\partial}{\partial y}+\frac{\partial}{\partial s}\right)A_{n+1}(x,y,s).
\end{equation}
By rewriting~\eqref{Anxys-recu} in terms of generating function $A:=A(x,y,s;z)$, we have
\begin{equation}\label{recu-Anxys02}
\frac{\partial}{\partial z}A=(s+y)A+xy\left(\frac{\partial}{\partial x}+\frac{\partial}{\partial y}+\frac{\partial}{\partial s}\right)A.
\end{equation}
It is routine to check that the generating function
$$\widetilde{A}=\mathrm{e}^{z(y+s)}\left(\frac{y-x}{y\mathrm{e}^{xz}-x\mathrm{e}^{yz}}\right)^2$$
satisfies~\eqref{recu-Anxys02}. Moreover, $\widetilde{A}(0,0,0;z)=1,\widetilde{A}(x,0,s;z)=\mathrm{e}^{sz}$ and $\widetilde{A}(0,y,s;z)=\mathrm{e}^{z(y+s)}$. The proof of $A=\widetilde{A}$ follows.

\quad (B)
Setting
$u=2xy,v=x+y,t=s+y$ and $I=LM$,
we get
$D_{G_5}(u)=uv,D_{G_5}(v)=u,D_{G_5}(t)=u$ and $D_{G_5}(I)=It$. Thus we get a new grammar
\begin{equation}\label{JAcobi-gram02}
G_6=\{I\rightarrow It,t\rightarrow u, u\rightarrow uv,v\rightarrow u\}.
\end{equation}
Note that $D_{G_6}(I)=It,~D_{G_6}^2(I)=I(t^2+u)$ and $D_{G_6}^3(I)=I(t^3+3tu+uv)$.
Then by induction, it is easy to verify that there exist nonnegative integers $\gamma_{n,i,j}$ such that
\begin{equation}\label{DG101}
D_{G_6}^n(I)=I\sum_{i=0}^nt^i\sum_{j=0}^{\lrf{(n-i)/2}}\gamma_{n,i,j}u^jv^{n-i-2j}.
\end{equation}
Then upon taking $u=2xy,v=x+y,t=s+y$ and $I=LM$, we get~\eqref{Anxys-gamma}.
In particular, $\gamma_{0,0,0}=1$ and $\gamma_{0,i,j}=0$ if $(i,j)\neq (0,0)$.
Since $D_{G_6}^{n+1}(I)=D_{G_6}\left(D_{G_6}^n(I)\right)$,
we obtain
\begin{align*}
D_{G_6}\left(D_{G_6}^n(I)\right)&
%&=D_{G_1}\left(I\sum_{i=0}^nt^i\sum_{j=0}^{\lrf{(n-i)/2}}2^j\gamma_{n,i,j}u^jv^{n-i-2j}\right)\\
=I\sum_{i,j}\gamma_{n,i,j}\left(t^{i+1}u^jv^{n-i-2j}+it^{i-1}u^{j+1}v^{n-i-2j}\right)+\\
&I\sum_{i,j}\gamma_{n,i,j}\left(jt^iu^jv^{n+1-i-2j}+(n-i-2j)t^iu^{j+1}v^{n-1-i-2j}\right).
\end{align*}
Comparing the coefficients of $t^iu^jv^{n+1-i-2j}$ in both sides of the above expansion, we get~\eqref{gammanij-recu}.

\quad (C) The combinatorial interpretation of $\gamma_{n,i,j}$ can be found by using the following grammatical labeling.
Given a 0-1-2 increasing rooted forest $T$, the root 0 is labeled by $I$. For the children of the root,
each child with degree 0 (a leaf of the root) is labeled by $t$ and each child with degree 1 is labeled by 1.
For the other vertices (not the children of the root), each leaf is labeled by $u$, each vertex with degree 1 is labeled by $v$ and each vertex of degree 2 is labeled by 1.
See Fig.~\ref{fig:simple} for an example, where
the grammatical labels are in parentheses.
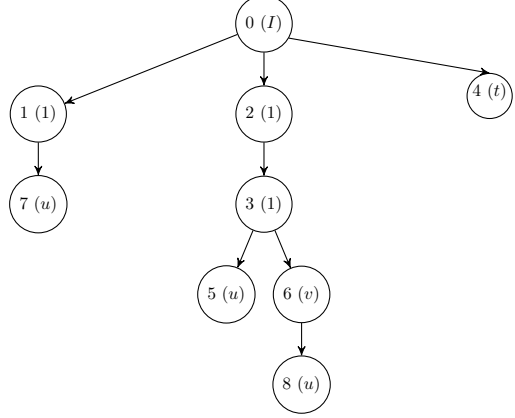
\begin{figure}[t]\label{fig:simple}
\centering

\caption{The labeling of a 0-1-2 increasing rooted forest on $\{0,1,2,\ldots,8\}$. \label{fig:simple}}
{
    \begin{tikzpicture}[->,>=stealth', level/.style={sibling distance = 5cm/#1, level distance = 2.0cm}, scale=0.6,transform shape]
    \node [treenode] {$0$~$(I)$}
    child
    {
        node [treenode] {$1$~$(1)$}
        child
            {
                node [treenode] {$7$~$(u)$}
            }}
    child
    {
        node [treenode] {$2$~(1)}
        child
        {
            node [treenode] {$3$~$(1)$}
            child
            {
                node [treenode] {$5$~$(u)$}
            }
            child
           {
                node [treenode] {$6$~$(v)$}
            child
            {
                node [treenode] {$8$~$(u)$}
            }}
        }
  %      child[edge from parent path ={(\tikzparentnode.-50) -- (\tikzchildnode.north)}]
%        {
%            node [treenode,yshift=0.4cm] (a) {}   % delay the text till later
%        }
    }
    child[edge from parent path ={(\tikzparentnode.-30) -- (\tikzchildnode.north)}]
    {
        node [treenode,yshift=0.4cm] (b) {}       % delay the text till later
    }
;
% ------------------------------------------------ put the text into subtree nodes
%\node[align=center,yshift=0.1cm] at (a) {$5$~$(u)$};
\node[align=center,yshift=0.1cm] at (b) {$4$~$(t)$};
\end{tikzpicture}
}
\end{figure}
Let $T$ be the 0-1-2 increasing rooted forest given in Fig.~\ref{fig:simple}.
Then there are four cases to be considered:
\begin{itemize}
\item [\rm ($i$)]If we add 9 as a child of the root 0, then the vertex 9 becomes a leaf of the root with label $t$.
This corresponds to the substitution rule $I\rightarrow It$;
\item [\rm ($ii$)]If we add 9 as a child of the vertex 4, the label $t$ of 4 becomes 1, and the vertex 9 gets the label $u$.
This corresponds to the substitution rule $t\rightarrow u$;
  \item [\rm ($iii$)]If we add 9 as a child of the vertex 5 (resp.~7, 8), the label $u$ of 5 (resp.~7, 8) becomes $v$, and the vertex 9 gets the label $u$.
This corresponds to the substitution rule $u\rightarrow uv$;
 \item [\rm ($iv$)]If we add 9 as a child of the vertex 6, the label $v$ of 6 becomes 1, and the vertex 9 gets the label $u$.
This corresponds to the substitution rule $v\rightarrow u$.
\end{itemize}

The aforementioned four cases exhaust all the cases to construct a 0-1-2 increasing rooted forest $T'$ on $\{0,1,2,\ldots,n,n+1\}$ from a 0-1-2 increasing rooted forest $T$ on
$\{0,1,2,\ldots,n\}$ by adding $n+1$ as a leaf.
Since $D_{G_6}^n(I)$ equals the sum of the weights
of 0-1-2 increasing rooted forests on $\{0,1,2,\ldots,n\}$, the coefficient $\gamma_{n,i,j}$ equals the number of 0-1-2 increasing rooted forest $T$ on $\{0,1,2,\ldots,n\}$ with $i+j$ leaves, among which $i$ leaves are the children of the root 0. This completes the proof.
%%%%%%%%%%%%%%%%%%%%%%%%%%%%%%%%%%%%%%%%%%%%%%%%%%%%%%%%%%%%%%%%%%%%%%%%%%
%%%%%%%%%%%%%%%%%%%%%%%%%%%%%%%%%%%%%%%%%%%%%%%%%%%%%%%%%%%%%%%%%%%%%%%%%%
%%%%%%%%%%%%%%%%%%%%%%%%%%%%%%%%%%%%%%%%%%%%%%%%%%%%%%%%%%%%%%%%%%%%%%%%%%
%%%%%%%%%%%%%%%%%%%%%%%%%%%%%%%%%%%%%%%%%%%
\section{The multivariate $k$-th order Eulerian polynomials}\label{Section3}
%%%%%%%%%%%%%%%%%%%%%%%%%%%%%%%%%%%%%%%%%%%
%%%%%%%%%%%%%%%%%%%%%%%%%%%%%%%%%%%%%%%%%%%%%%%%%%%%%%%%%%%%%%%%%%%%%%%%%%%%%%%%%
%%%%%%%%%%%%%%%%%%%%%%%%%%%%%%%%%%%%%%%%%%%%%%%%%%%%%%%%%%%%%%%%%%%%%%%%%%%%%%%%%
%%%%%%%%%%%%%%%%%%%%%%%%%%%%%%%%%%%%%%%%%%%%%%%%%%%%%%%%%%%%%%%%%%%%%%%%%%%%%%%%%
%%%%%%%%%%%%%%%%%%%%%%%%%%%%%%%%%%%%%%%%%%%%%%%%%%
%%%%%%%%%%%%%%%%%%%%%%%%%%%%%%%%%%%%%%%%%%%%
%%%%%%%%%%%%%%%%%%%%%%%%%%%%%%%%%%%%%%%%%%%%
\subsection{Preliminary}
%%%%%%%%%%%%%%%%%%%%%%%%%%%%%
\hspace*{\parindent}
%%\subsection{%%%%%%%%%%%%%%%%%%%%%%%%%%%%%%%%%%%%%%%%%%%%%%%%%%%
%%%%%%%%%%%%%%%%%%%%%%%%%%%%%%%%%%%%%%%%%%%%%%%%%%%%%%%%%%%

The {\it second-order Eulerian polynomials} are defined by
$$C_n(x)=(1-x)^{2n+1}\sum_{k=0}^\infty \Stirling{n+k}{k}x^k.$$
In order to find a combinatorial interpretation of the coefficients of $C_n(x)$ in terms of
descents of permutations, Gessel and Stanley~\cite{Gessel78} introduced
Stirling permutations.
A {\it Stirling permutation} of order $n$ is a permutation of $\{1,1,2,2,\ldots,n,n\}$ such that
for each $i$, $1\leqslant i\leqslant n$, all entries between the two occurrences of $i$ are larger than $i$.
Denote by $\mqn$ the set of {\it Stirling permutations} of order $n$.
Let $\sigma=\sigma_1\sigma_2\cdots\sigma_{2n}\in\mqn$. In this paper, we always set $\sigma_0=\sigma_{2n+1}=0$.
Following~\cite{Bona08,Gessel78}, for $0\leqslant i\leqslant 2n$, we say that an index $i$ is a {\it descent} (resp.~{\it ascent}, {\it plateau}) of $\sigma$ if
$\sigma_{i}>\sigma_{i+1}$ (resp. $\sigma_{i}<\sigma_{i+1}$, $\sigma_{i}=\sigma_{i+1}$).
Let $\des(\sigma),\asc(\sigma)$ and $\plat(\sigma)$ be the number of descents, ascents and plateaux of $\sigma$, respectively.
According to~\cite[Theorem~2.1]{Gessel78}, one has
$$C_n(x)=\sum_{\sigma\in\mqn}x^{\des(\sigma)}=\sum_{j=1}^nC_{n,j}x^j.$$
where $C_{n,j}$ is called the {\it second-order Eulerian number}.
Below are $C_n(x)$ for $n\leqslant 5$:
\begin{align*}
C_1(x)&=x,~C_2(x)=x+2x^2,~
C_3(x)=x+8x^2+6x^3,\\
C_4(x)&=x+22x^2+58x^3+24x^4,~
C_5(x)=x+52x^2+328x^3+444x^4+120x^5.
\end{align*}

The {\it trivariate second-order Eulerian polynomials} are defined as follows:
$$C_n(x,y,z)=\sum_{\sigma\in\mqn}x^{\asc{(\sigma)}}y^{\des(\sigma)}z^{\plat{(\sigma)}}.$$
It is now well known that
\begin{equation}\label{Dumont80}
C_{n+1}(x,y,z)=xyz\left(\frac{\partial}{\partial x}+\frac{\partial}{\partial y}+\frac{\partial}{\partial z}\right)C_n(x,y,z),~C_0(x,y,z)=1.
\end{equation}
As pointed out by Chen and Fu~\cite{Chen22}, the recursion~\eqref{Dumont80} first appeared in the work of Dumont~\cite[p.~317]{Dumont80}, which implies that
$C_n(x,y,z)$ is symmetric in the variables $x,y$ and $z$.
The symmetry of $C_n(x,y,z)$ was rediscovered by Janson~\cite[Theorem~2.1]{Janson08} by constructing an urn model.
In~\cite{Haglund12}, Haglund and Visontai introduced a refinement of the polynomial $C_n(x,y,z)$ by indexing each ascent, descent and plateau by the value where they appear.
%The symmetry of $C_n(x,y,z)$ also follows from the symmetry of the recursion~\cite[Eq.~(17)]{Haglund12}.

Let $G_7$ be the following grammar
\begin{equation}\label{eq001}
G_7=\{x \rightarrow xyz, y\rightarrow xyz, z\rightarrow xyz\}.
\end{equation}
It has been shown by Dumont~\cite{Dumont80}, Chen etal.~\cite{Chen2102} and Ma etal.~\cite{Ma19} that
$$D_{G_7}^n(x)=C_n(x,y,z).$$

The following definition will be used in the following discussion.
\begin{definition}
A 0-1-2-$\cdots$-k increasing plane tree on $[n]$ is an increasing plane tree with each
vertex with at most $k$ children.
\end{definition}

Using the change of variables~\eqref{grammars02}, we see that
$D_{G_7}(u)=3w,~D_{{G_7}}(v)=2uw,~D_{G_7}(w)=vw$, which yields a new grammar
\begin{equation}\label{G8def}
G_8=\{u\rightarrow 3w, v\rightarrow 2uw,~w\rightarrow vw\}.
\end{equation}

Chen and Fu~\cite{Chen22} gave an interpretation of the grammar $G_8$ and obtained the following result.
\begin{theorem}[{\cite{Chen22}}]\label{thm01}
For $n\geqslant 1$, one has
$$C_n(x,y,z)=\sum_{k\geqslant 1}(xyz)^k\sum_{j\geqslant 0}\gamma_{n,k,j}(xy+yz+zx)^{j}(x+y+z)^{2n+1-2j-3k},$$
where the coefficient $\gamma_{n,k,j}$ equals the number of 0-1-2-3 increasing plane trees
on $[n]$ with $k$ leaves, $j$ degree one vertices and $i$ degree two vertices.
\end{theorem}

\begin{corollary}
For $n\geqslant 1$, one has
$$C_n(x)=\sum_{k\geqslant 1}x^k\sum_{j\geqslant 0}\gamma_{n,k,j}(1+2x)^{j}(2+x)^{2n+1-2j-3k}.$$
\end{corollary}

In the next section, we shall consider a decomposition of the statistic $\plat$ of $k$-Stirling permutations, i.e., write $\plat$ as a sum of the numbers of $j$-plateaux.
%%%%%%%%%%%%%%%%%%%%%%%%%%%%%%%%%%%%%%%%%%%%%%%%%%%%%%%%%%%%%%%%%%%%%%%%%%%%%%%%%
%%%%%%%%%%%%%%%%%%%%%%%%%%%%%%%%%%%%%%%%%%%%%%%%%%
%%%%%%%%%%%%%%%%%%%%%%%%%%%%%%%%%%%%%%%%%%%%
%%%%%%%%%%%%%%%%%%%%%%%%%%%%%%%%%%%%%%%%%%%%
\subsection{A key Lemma}
%%%%%%%%%%%%%%%%%%%%%%%%%%%%%
\hspace*{\parindent}
%%\subsection{%%%%%%%%%%%%%%%%%%%%%%%%%%%%%%%%%%%%%%%%%%%%%%%%%%%
%%%%%%%%%%%%%%%%%%%%%%%%%%%%%%%%%%%%%%%%%%%%%%%%%%%%%%%%%%%

Let $k$ be a given positive integer, and let $j^k$ denote $k$ times of the letter $j$. A $k$-Stirling permutation of order $n$ is a multiset permutation of $\{1^k,2^k,\ldots,n^k\}$
with the property that all elements between two occurrences of $i$ are at least $i$, where $i\in [n]$.
Let $\mqn(k)$ be the set of $k$-Stirling permutations of order $n$. Clearly, $\mqn(1)=\msn$ and $\mqn(2)=\mqn$.
Following~\cite[p.~657]{Stanley2011}, an {\it $k$-ary tree} $T$ is
either empty, or else one specially designated vertex is called the root of $T$ and
the remaining vertices (excluding the root) are put into a (weak) ordered partition
$(T_1,\ldots,T_k)$ of exactly $k$ disjoint (possibly empty) sets $T_1,\ldots, T_k$, each of which is an
$k$-ary tree. A bijection between $\mqn(k)$ and the set of
$(k+1)$-ary increasing trees was independently established by Gessel~\cite{Park1994} and Janson-Kuba~\cite[Theorem~1]{Janson11}.

Let $\sigma\in\mqn(k)$.
The ascents, descents and plateaux of $\sigma$ of are defined as before, where we always set $\sigma_0=\sigma_{kn+1}=0$.
More precisely, an index $i$ is called an ascent (resp. descent, plateau) of $\sigma$ if $\sigma_i<\sigma_{i+1}$ (resp. $\sigma_i>\sigma_{i+1}$, $\sigma_i=\sigma_{i+1}$).
It is clear that $\asc(\sigma)+\des(\sigma)+\plat(\sigma)=kn+1$.
As a natural refinement of ascents, descents and plateaux, Janson and Kuba~{\cite{Janson11} introduced the following definition, and
related the distribution of $j$-ascents, $j$-descents and $j$-plateaux in $k$-Stirling permutations
with certain parameters in $(k+1)$-ary increasing trees.
\begin{definition}[{\cite{Janson11}}]\label{Lemma-Stirling}
An index $i$ is called a $j$-plateau (resp.~$j$-descent,~$j$-ascent) if $i$ is a plateau (resp.~descent,~ascent) and there are exactly $j-1$ indices $\ell<i$ such that
$a_{\ell}=a_i$.
\end{definition}
Let $\operatorname{plat}_j(\sigma)$ be the number of $j$-plateaux of $\sigma$. For $\sigma\in\mqn(k)$, it is clear that
$\operatorname{plat}_j(\sigma)\leqslant k-1$.
\begin{example}
Consider the $4$-Stirling permutation $\sigma=111223333221$.
The set of $1$-plateaux is given by $\{1,4,6\}$, the set of $2$-plateaux is given by $\{2,7\}$, and the set of $3$-plateaux is given by $\{8,10\}$.
Thus $\operatorname{plat}_1(\sigma)=3$ and $\operatorname{plat}_2(\sigma)=\operatorname{plat}_3(\sigma)=2$.
\end{example}

The {\it multivariate $k$-th order Eulerian polynomials} $C_n(x_1,\ldots,x_{k+1})$ are defined by
$$C_n(x_1,x_2,\ldots,x_{k+1})=\sum_{\sigma\in\mqn(k)}{x_1}^{\operatorname{plat}_1(\sigma)}{x_2}^{\operatorname{plat}_2(\sigma)}\cdots {x_{k-1}}^{\operatorname{plat}_{k-1}(\sigma)}{x_{k}}^{\operatorname{des}(\sigma)}{x_{k+1}}^{\operatorname{asc}(\sigma)}.$$
In particular, when $x_1=z,~x_2=\cdots=x_{k-1}=0$, $x_k=y$ and $x_{k+1}=x$, the polynomials $C_n(x_1,x_2,\ldots,x_{k+1})$ reduce to
$C_n(x,y,z)$; when $x_1=x_2=\cdots=x_{k-1}=0$, $x_k=1$ and $x_{k+1}=x$, the polynomials $C_n(x_1,x_2,\ldots,x_{k+1})$ reduce to
$A_n(x)$.

In the following discussion,
we always let $\XX_{k+1}=\{x_1,x_2,\ldots,x_{k+1}\}$ and let $e_i$ be the $i$-th
elementary symmetric function associated with $\XX_{k+1}$. In particular, $$e_0=1,~e_1=x_1+x_2+\cdots+x_{k+1},~e_k=\sum_{i=1}^k\frac{e_{k+1}}{x_i},~e_{k+1}=x_1x_2\cdots x_{k+1}.$$
The following lemma is fundamental.
\begin{lemma}\label{Stirling}
Let
$G_9=\{x_1 \rightarrow e_{k+1},~ x_2\rightarrow e_{k+1},\ldots,~ x_{k+1}\rightarrow e_{k+1}\}$,
where $e_{k+1}=x_1x_2\cdots x_{k+1}$.
For $n\geqslant 1$, one has $$D_{G_9}^n(x_1)=C_n(x_1,x_2,\ldots,x_{k+1}).$$
\end{lemma}
\begin{proof}
We shall show
that the grammar $G_9$ can be used to generate $k$-Stirling permutations.
We first introduce a grammatical labeling of $\sigma\in \mqn(k)$ as follows:
\begin{itemize}
  \item [\rm ($L_1$)]If $i$ is an ascent, then put a superscript label $x_{k+1}$ right after $\sigma_i$;
 \item [\rm ($L_2$)]If $i$ is a descent, then put a superscript label $x_k$ right after $\sigma_i$;
\item [\rm ($L_3$)]If $i$ is a $j$-plateau, then put a superscript label $x_j$ right after $\sigma_i$.
\end{itemize}
The weight of $\sigma$ is
defined as the product of the labels, that is $$w(\sigma)={x_1}^{\operatorname{plat}_1(\sigma)}{x_2}^{\operatorname{plat}_2(\sigma)}\cdots {x_{k-1}}^{\operatorname{plat}_{k-1}(\sigma)}{x_{k}}^{\operatorname{des}(\sigma)}{x_{k+1}}^{\operatorname{asc}(\sigma)}.$$
Recall that we always set $\sigma_0=\sigma_{kn+1}=0$.
Thus the index $0$ is always an ascent and the index $kn$ is always a descent.
Thus $\mq_1(k)=\{^{x_{k+1}}1^{x_1}1^{x_2}1^{x_3}\cdots 1^{x_k}\}$. The are $k+1$ elements in $\mq_2(k)$ and they can be labeled as follows, respectively:
$$^{x_{k+1}}1^{x_1}1^{x_2}\cdots 1^{x_{k-1}}1^{x_{k+1}}2^{x_1}2^{x_2}\cdots 2^{x_{k-1}}2^{x_k},$$
$$^{x_{k+1}}1^{x_1}1^{x_2}\cdots 1^{x_{k-2}}1^{x_{k+1}}2^{x_1}2^{x_2}\cdots 2^{x_{k-1}}2^{x_k}1^{x_k},$$
$$\cdots$$
$$^{x_{k+1}}2^{x_1}2^{x_2}\cdots2^{x_{k-1}}2^{x_k}1^{x_1}1^{x_2}\cdots 1^{x_{k-1}}1^{x_k}.$$
Note that $D_{G_9}(x_1)=e_{k+1}$ and $D_{G_9}^2(x_1)=e_ke_{k+1}$.
Then the weight of the element in $\mq_1(k)$ is given by $D_{G_9}(x_1)$, and the sum of weights of the elements in $\mq_2(k)$ is given by $D_{G_9}^2(x)$.
Hence the result holds for $n=1,2$.
We proceed by induction on $n$.
Suppose we get all labeled permutations in $\mq_{n-1}(k)$, where $n\geqslant 3$. Let
$\sigma'$ be obtained from $\sigma\in \mq_{n-1}(k)$ by inserting the string $nn\cdots n$ with length $k$.
Then the changes of labeling are illustrated as follows:
$$\cdots\sigma_i^{x_j}\sigma_{i+1}\cdots\mapsto \cdots\sigma_i^{x_{k+1}}n^{x_1}n^{x_2}\cdots n^{x_k}\sigma_{i+1}\cdots;$$
$$\sigma^{x_k}\mapsto \sigma^{x_{k+1}}n^{x_1}n^{x_2}\cdots n^{x_k};~~\quad ^{x_{k+1}}\sigma \mapsto ^{x_{k+1}}n^{x_1}n^{x_2}\cdots n^{x_k}\sigma.$$
In each case, the insertion of the string $nn\cdots n$ corresponds to one substitution rule in $G_9$.
Then the action of $D_{G_9}$ on the set of weights of all elements in $\mq_{n-1}(k)$ gives the set of weights of all elements in $\mq_n(k)$.
Therefore, we get a grammatical interpretation of $C_n(x_1,x_2,\ldots,x_{k+1})$, and this completes the proof.
\end{proof}

From the symmetry of the grammar $G_9$ and $D_{G_9}(x_1)=e_{k+1}$, we get the
following result.
\begin{corollary}\label{corJan}
The multivariate polynomials $C_n(x_1,x_2,\ldots,x_{k+1})$ are symmetric, i.e., the variables are exchangeable.
\end{corollary}

By combining an urn model for the exterior leaves of $(k+1)$-ary increasing trees and a bijection between $(k+1)$-ary increasing trees and $k$-Stirling permutations,
Janson, Kuba and Panholzer~\cite[Theorem~2,~Theorem~8]{Janson11} found that the variables in $C_n(x_1,x_2,\ldots,x_{k+1})$ are exchangeable. Therefore, in an equivalent form, Corollary~\ref{corJan} was first obtained in~\cite{Janson11}. It should be noted that in~\cite{Janson11}, there is no explicit connection to the $k$-th order
Eulerian polynomials is brought up.
%%%%%%%%%%%%%%%%%%%%%%%%%%%%%%%%%%%%%%%%%%%%%%%%%%%%%%%%%%%%%%%%%%%%%%%%%%%%%%%%%
%%%%%%%%%%%%%%%%%%%%%%%%%%%%%%%%%%%%%%%%%%%%%%%%%%
%%%%%%%%%%%%%%%%%%%%%%%%%%%%%%%%%%%%%%%%%%%%
%%%%%%%%%%%%%%%%%%%%%%%%%%%%%%%%%%%%%%%%%%%%
\subsection{Main results}
%%%%%%%%%%%%%%%%%%%%%%%%%%%%%
\hspace*{\parindent}
%%\subsection{%%%%%%%%%%%%%%%%%%%%%%%%%%%%%%%%%%%%%%%%%%%%%%%%%%%
%%%%%%%%%%%%%%%%%%%%%%%%%%%%%%%%%%%%%%%%%%%%%%%%%%%%%%%%%%%

\begin{theorem}\label{mainthm}
For $n\geqslant 2$ and $k\geqslant n-2$, we have
\begin{equation}\label{main01}
C_n(x_1,x_2,\ldots,x_{k+1})=\sum \gamma(n;i_1,i_2,\ldots,i_n)e_{k-n+2}^{i_n}e_{k-n+3}^{i_{n-1}}\cdots e_k^{i_2} e_{k+1}^{i_1},
\end{equation}
where the summation is over all sequences $(i_1,i_2,\ldots,i_{n})$ of nonnegative integers such that
$i_1+i_2+\cdots+i_{n}=n$, $1\leqslant i_1\leqslant n-1$, $i_n=0$ or $i_n=1$. When $i_n=1$, one has $i_1=n-1$.
The coefficients $\gamma(n;i_1,i_2,\ldots,i_n)$ equals the number of 0-1-2-$\cdots$-k-(k+1) increasing plane
trees on $[n]$ with $i_j$ degree $j-1$ vertices for all $1\leqslant j\leqslant n$.
\end{theorem}
\begin{proof}
Let $G_9$ be the grammar given in Lemma~\ref{Stirling}. We first consider a change of $G_9$.
Note that $D_{G_9}(x_1)=e_{k+1},~D_{G_9}(e_{i})=(k-i+2)e_{i-1}e_{k+1}$ for $1\leqslant i\leqslant k+1$. Thus we get a new grammar
\begin{equation}\label{grammar002}
G_{10}=\{x_1\rightarrow e_{k+1},~e_{i}\rightarrow (k-i+2)e_{i-1}e_{k+1}~\text{for $1\leqslant i\leqslant k+1$}\},
\end{equation}
Note that $G_{10}(x_1)=e_{k+1},~G_{10}^2(x_1)=e_ke_{k+1},~G_{10}^3(x_1)=e_k^2e_{k+1}+2e_{k-1}e_{k+1}^2$.
% $D_{G_2}^4(x_1)=e_k^3e_{k+1}+8e_{k-1}e_ke_{k+1}^2+6e_{k-2}e_{k+1}^3$ and
%$D_{G_2}^5(x_1)=e_k^4e_{k+1}+22e_k^2e_{k-1}e_{k+1}^2+16e_{k-1}^2e_{k+1}^3+42e_{k-2}e_ke_{k+1}^3+24e_{k-3}e_{k+1}^4$.
By induction, we assume that
\begin{equation}\label{DG2x1}
G_{10}^n(x_1)=\sum \gamma(n;i_1,i_2,\ldots,i_n)e_{k-n+2}^{i_n}e_{k-n+3}^{i_{n-1}}\cdots e_k^{i_2} e_{k+1}^{i_1}.
\end{equation}
Note that
\begin{align*}
G_{10}^{n+1}(x_1)&=G_{10}\left(\sum \gamma(n;i_1,i_2,\ldots,i_n)e_{k-n+2}^{i_n}e_{k-n+3}^{i_{n-1}}\cdots e_k^{i_2} e_{k+1}^{i_1}\right)\\
&=\sum ni_n\gamma(n;i_1,i_2,\ldots,i_n)e_{k-n+1}e_{k-n+2}^{i_n-1}e_{k-n+3}^{i_{n-1}}\cdots e_k^{i_2} e_{k+1}^{i_1+1}+\\
&\sum (n-1)i_{n-1}\gamma(n;i_1,i_2,\ldots,i_n)e_{k-n+2}^{i_n+1}e_{k-n+3}^{i_{n-1}-1}\cdots e_k^{i_2} e_{k+1}^{i_1+1}+\cdots+\\
&\sum 2i_2\gamma(n;i_1,i_2,\ldots,i_n)e_{k-n+2}^{i_n}e_{k-n+3}^{i_{n-1}}\cdots e_{k-1}^{i_3+1}e_k^{i_2-1} e_{k+1}^{i_1+1}+\\
&\sum i_1\gamma(n;i_1,i_2,\ldots,i_n)e_{k-n+2}^{i_n}e_{k-n+3}^{i_{n-1}}\cdots e_k^{i_2+1} e_{k+1}^{i_1},
\end{align*}
which yields that the expansion~\eqref{DG2x1} holds for $n+1$. Combining Lemma~\ref{Stirling} and~\eqref{DG2x1}, we get~\eqref{main01}.
By induction, one can easily verify that $i_1+i_2+\cdots+i_n=n$, $1\leqslant i_1\leqslant n-1$, $i_n=1$ or $i_n=0$.

By using~\eqref{grammar002}, the combinatorial interpretation of
the coefficients $\gamma(n;i_1,i_2,\ldots,i_n)$ can be proved along the same lines as the proof of~\cite[Theorem~4.1]{Chen22}.
However, we give a direct proof of it for our purpose.
Let $T$ be a 0-1-2-$\cdots$-k-(k+1) increasing plane tree on $[n]$. The labeling of
$T$ is given by labeling a degree $i$ vertex by $e_{k-i+1}$ for all $0\leqslant i\leqslant k+1$.
In particular, label a leaf by $e_{k+1}$ and label a degree $k+1$ vertex by $1$.
Let $T'$ be a 0-1-2-$\cdots$-k-(k+1)
increasing plane tree on $[n+1]$ by adding $n+1$ to $T$ as a leaf.
We can add $n+1$ to $T$ only as a child of a vertex $v$ that is not of degree $k+1$.
For $1\leqslant i\leqslant k+1$, if the vertex $v$ is a degree $k-i+1$ vertex with label $e_i$,
there are $k-i+2$ cases to attach $n+1$ (from left to right, say). In either case, in $T'$, the vertex $v$ becomes a
degree $k-i+2$ with label $e_{i-1}$ and $n+1$ becomes a leaf with label $e_{k+1}$. Hence
the insertion of $n+1$ corresponds to the substitution rule $e_{i}\rightarrow (k-i+2)e_{i-1}e_{k+1}$.
Therefore, $G_{10}(x_1)$ equals the sum of the weights
of 0-1-2-$\cdots$-(k+1) increasing plane trees on $[n]$, and the combinatorial interpretation of $\gamma(n;i_1,i_2,\ldots,i_n)$ follows. This completes the proof.
\end{proof}

For convenience, we present the expansions of $G_{10}^4(x_1)$ and $G_{10}^5(x_1)$:
\begin{align*}
&G_{10}^4(x_1)=e_k^3e_{k+1}+8e_{k-1}e_ke_{k+1}^2+6e_{k-2}e_{k+1}^3,\\
&G_{10}^5(x_1)=e_k^4e_{k+1}+22e_k^2e_{k-1}e_{k+1}^2+16e_{k-1}^2e_{k+1}^3+42e_{k-2}e_ke_{k+1}^3+24e_{k-3}e_{k+1}^4.
\end{align*}

By using $G_{10}^{n+1}(x_1)=G_{10}\left(G_{10}^{n}(x_1)\right)$, it is routine to verify that
\begin{align*}
&\gamma(n+1;1,n,0\ldots,0)=\gamma(n;1,n-1,0,\ldots,0)=1,\\
&\gamma(n+1;n,0,\ldots,0,1)=n\gamma(n;n-1,0,\ldots,0,1)=n!,\\
&\gamma(n+1;i_1,i_2,\ldots,i_n,0)=i_1\gamma(n;i_1,i_2-1,i_3,\ldots,i_n)+\\
&\sum_{j=2}^{n-1}j(i_j+1)\gamma(n;i_1-1,i_2,\ldots,i_{j-1},i_j+1,i_{j+1}-1,i_{j+2}\ldots,i_n).
\end{align*}

Note that $\gamma(3;2,0,1,0,\ldots,0)=2,~\gamma(4;2,1,1,0,\ldots,0)=8$ and
$$\gamma(n+1;2,n-2,1,0,\ldots,0)=2\gamma(n;2,n-3,1,0,\ldots,0)+2(n-1)\gamma(n;1,n-1,0,\ldots,0).$$
By induction, it is easy to verify that
\begin{equation}\label{gamman3}
\gamma(n;2,n-3,1,0,\ldots,0)=2^n-2n~\text{for $n\geqslant 3$}.
\end{equation}
Recall that the second-order Eulerian numbers $C_{n,j}$ satisfy the recurrence relation
\begin{equation*}\label{secondEu-recu}
C_{n+1,j}=jC_{n,j}+(2n+2-j)C_{n,j-1},
\end{equation*}
with the initial conditions $C_{1,1}=1$ and $C_{1,j}=0$ if $j\neq 1$ (see~\cite{Bona08,Gessel78}). In particular,
\begin{equation*}
C_{n,2}=2^{n+1}-2(n+1).
\end{equation*}
Comparing this with~\eqref{gamman3}, we see that
$\gamma(n;2,n-3,1,0,\ldots,0)=C_{n-1,2}$ for $n\geqslant 3$.
Moreover, $\gamma(n+1;n,0,\ldots,0,1)=C_{n,n}=n!$.
Following Janson~\cite{Janson08}, the number $C_{n,j}$ equals the number of increasing plane trees on $[n+1]$ with $k$ leaves.
So we get the following result.
\begin{corollary}
For $n\geqslant 2$ and $1\leqslant j\leqslant n-1$, we have
$$C_{n-1,j}=\sum_{i_2+i_3+\cdots+i_{n}=n-j} \gamma(n;j,i_2,\ldots,i_{n-1},i_n).$$
\end{corollary}

\end{document}